\documentclass[a4j,12pt,final]{article}
\usepackage{amsmath}
\usepackage{amssymb}
\usepackage{amsthm}
\usepackage{amscd}
\usepackage{bm}

\newtheorem{theorem}{Theorem}[section]
\newtheorem{lemma}{Lemma}[section]
\theoremstyle{definition}
\newtheorem{definition}{Definition}[section]
\theoremstyle{remark}

\theoremstyle{proposition}

\numberwithin{equation}{section}

\theoremstyle{corollary}
\newtheorem{corollary}{Corollary}[section]

\theoremstyle{conjecture}

\begin{document}

\title{On the Continuous Cohomology of\\
 a semi-direct product Lie group}
\author{Naoya Suzuki}
\date{}
\maketitle

\begin{abstract}
Let $G$ be a Lie group and $H$ be a subgroup of it. We can construct a bisimplicial manifold $NG(*) \rtimes NH(*)$ and 
the de Rham complex $\Omega^*(NG(*) \rtimes NH(*))$ on it. This complex is a triple complex and the cohomology of its
total complex is isomorphic to $H^*(B(G \rtimes H))$.
In this paper, we show that the total complex of the double complex $\Omega^q(NG(*) \rtimes NH(*))$ is 
isomorphic to the continuous cohomology $H_c^*(G \rtimes H;S^q{\mathcal G}^* \otimes S^q{\mathcal H}^*)$ for any fixed $q$.
\end{abstract}

\section{Introduction}
Let $G$ be a Lie group. In the theory of simplicial manifold, there is a well-known simplicial manifold $NG$ called nerve of $G$.
The de Rham complex $\Omega^*(NG(*))$ on it is a double complex, and the cohomology of its total complex is 
isomorphic to $H^*(BG)$. In \cite{Bot}, Bott proved the cohomology of its horizontal complex $\Omega^p(NG(*))$ is
isomorphic to the continuous cohomology $H_c^*(G;S^q{\mathcal G}^*)$ for any fixed $q$.

On the other hand, for a subgroup $H$ of $G$ we can construct a bisimplicial manifold $NG(*) \rtimes NH(*)$ and 
the de Rham complex $\Omega^*(NG(*) \rtimes NH(*))$ on it. This complex is a triple complex and the cohomology of its
total complex is isomorphic to $H^*(B(G \rtimes H))$ \cite{Suz}. 

In this paper, we show that the total complex of the double complex $\Omega^q(NG(*) \rtimes NH(*))$ is 
isomorphic to the continuous cohomology $H_c^*(G \rtimes H;S^q{\mathcal G}^*\otimes S^q{\mathcal H}^*)$ for any fixed $q$.

\section{Review of the simplicial de Rham complex}
In this section we recall the relation between the simplicial manifold $NG$ and the classifying space $BG$. We also recall
the notion of the equivariant version of the simplicial de Rham complex.

\subsection{The double complex on simplicial manifold}

For any Lie group $G$, we have simplicial manifolds $NG$, $PG$ and simplicial $G$-bundle  $\gamma : PG \rightarrow NG$
as follows:\\
\par
$NG(q)  = \overbrace{G \times \cdots \times G }^{q-times}  \ni (g_1 , \cdots , g_q ) :$  \\
face operators \enspace ${\varepsilon}_{i} : NG(q) \rightarrow NG(q-1)  $
$$
{\varepsilon}_{i}(g_1 , \cdots , g_q )=\begin{cases}
(g_2 , \cdots , g_q )  &  i=0 \\
(g_1 , \cdots ,g_i g_{i+1} , \cdots , g_q )  &  i=1 , \cdots , q-1 \\
(g_1 , \cdots , g_{q-1} )  &  i=q
\end{cases}
$$

\par
\medskip
$PG (q) = \overbrace{ G \times \cdots \times G }^{q+1 - times} \ni (\bar{g}_1 , \cdots , \bar{g}_{q+1} ) :$ \\
face operators \enspace $ \bar{\varepsilon}_{i} : PG(q) \rightarrow PG(q-1)  $ 
$$ \bar{{\varepsilon}} _{i} (\bar{g}_1 , \cdots , \bar{g}_{q+1} ) = (\bar{g}_1 , \cdots , \bar{g}_{i} , \bar{g}_{i+2}, \cdots , \bar{g}_{q+1})  \qquad i=0,1, \cdots ,q $$

\par
\medskip

We define $\gamma : PG \rightarrow NG $ as $ \gamma (\bar{g}_1 , \cdots , \bar{g}_{q+1} ) = (\bar{g}_1 {\bar{g}_2}^{-1} , \cdots , \bar{g}_{q} {\bar{g}_{q+1}}^{-1} )$.\\

For any simplicial manifold $\{ X_* \}$, we can associate a topological space $\parallel X_* \parallel $ 
called the fat realization defined as follows:
$$  \parallel X_* \parallel \enspace \buildrel \mathrm{def} \over = \coprod _{n}  {\Delta}^{n} \times X_n / \enspace ( {\varepsilon}^{i} t , x) \sim (  t , {\varepsilon}_{i} x).$$
Here ${\Delta}^{n}$ is the standard $n$-simplex and ${\varepsilon}^{i}$ is a face map of it.
It is well-known that 
$\parallel \gamma \parallel : \parallel PG \parallel \rightarrow \parallel NG \parallel$ is the universal bundle $EG \rightarrow BG$  (see \cite{Dup2} 
\cite{Mos} \cite{Seg}, for instance). \\

Now we introduce a double complex on a simplicial manifold.

\begin{definition}
For any simplicial manifold $ \{ X_* \}$ with face operators $\{ {\varepsilon}_* \}$, we have a double complex ${\Omega}^{p,q} (X) := {\Omega}^{q} (X_p) $ with derivatives as follows:
$$ \delta := \sum _{i=0} ^{p+1} (-1)^{i} {\varepsilon}_{i} ^{*}  , \qquad  d' := (-1)^{p} \times {\rm the \enspace exterior \enspace differential \enspace on \enspace }{ \Omega ^*(X_p) } .$$

\end{definition}

For $NG$ and $PG $ the following holds.

\begin{theorem}[\cite{Bot2} \cite{Dup2} \cite{Mos}]
 There exist ring isomorphisms 
$$ H^*({\Omega}^{*} (NG))  \cong  H^{*} (BG ), \qquad  H^*({\Omega}^{*} (PG)) \cong H^{*} (EG ).  $$
 Here ${\Omega}^{*} (NG)$  and  ${\Omega}^{*} (PG)$  mean the total complexes.
\end{theorem}

\subsection{Equivariant version}

When a Lie group $H$ acts on a manifold $M$, there is the complex of equivariant differential forms 
${\Omega}_H ^{*} (M) := ( {\Omega} ^{*} (M) \otimes S\mathcal{H}^*)^H$ with suitable differential $d_H$ (\cite{Ber} \cite{Car}).
Here $\mathcal{H}$ is the
Lie algebra of $H$ and
$S\mathcal{H}^*$
is the algebra of polynomial functions on $\mathcal{H}$. This is
called the Cartan Model.
When $M$ is a Lie group $G$, we can define a double complex ${\Omega}^{*} _H (NG(*))$ below in the same way as in Definition 2.1.

$$
\begin{CD}
{\Omega}^{p}_H (G ) \\
@AA{-d_H}A \\
{\Omega}^{p-1}_H (G )@>{{\varepsilon}_{0} ^{*} - {\varepsilon}_{1} ^{*} +{\varepsilon}_{2} ^{*} }>>{\Omega}^{p-1}_H (NG(2))\\
@.@AA{d_H}A\\
@.{\Omega}^{p-2}_H (NG(2))\\
@.@. \ddots \\
@.@.@.{\Omega}^{1}_H (NG(p)) \\
@.@.@.@AA{(-1)^p d_H }A\\
@.@.@.{\Omega}^{0}_H (NG(p))@>{ \sum _{i=0} ^{p+1} (-1)^{i} {\varepsilon}_{i} ^{*}}>> {\Omega}^{0}_H (NG(p+1)) 
\end{CD}
$$

\section{The cohomology of the horizontal complex}

At first, we recall the description of the cohomology of groups in terms of resolutions due to Hochschild and Mostow \cite{Ho}.

\begin{theorem}[\cite{Ho}]
If $G$ is a topological group and $M$ is a topological $G$-module, then the continuous cohomology
$H_c(G;M)$ is isomorphic to the cohomology of the invariant complex
$$ {\rm Inv}_G M \rightarrow {\rm Inv}_G X_0 \rightarrow {\rm Inv}_GX_1 \rightarrow \cdots$$
for any continuously injective resolution $M \rightarrow X_0 \rightarrow X_1 \rightarrow \cdots$ of $M$.
\end{theorem}

Now we recall the result of Bott in \cite{Bot}, which gives the cohomology of the horizontal complex of $\Omega^*(NG)$.

\begin{theorem}[Bott,\cite{Bot}]
For any fixed $q$,
$$H^{p+q}_{\delta}(\Omega^q(NG)) \cong H^{p}_c(G;S^q{\mathcal{G}^*}).$$
Here $\mathcal{G}^*$ is a ${\mathbb R}$-module of left-invariant $1$-forms on $G$.
\end{theorem}

\begin{proof}
Let $\bm{n}$ denote the ordered set $\{ 0<1<\cdots<n\}$ and $\bm{n}^{\natural}$ the underlying set of it.
We define $C{\mathbb Z}(\bm{n}) :=  {\mathbb Z}(\bm{n}^{\natural})$ as a free group generated by $\bm{n}^{\natural}$
then we have a natural arrow
$$r : C{\mathbb Z} \rightarrow {\mathbb Z}$$
defined by
$$r(\bm{n}) \left( \sum_{\alpha = 0, \cdots,n} a_{\alpha} \alpha \right) = \sum a_{\alpha}, ~~~a_{\alpha} \in {\mathbb Z}.$$
Bott called the kernel of $r$ the suspension of ${\mathbb Z}$ and denote it  $\Sigma {\mathbb Z}$.

We define  the suspension of $\mathcal{G}^*$ as $\Sigma \mathcal{G}^* := C{\mathbb Z} \otimes \mathcal{G}^*$. Then there exists the following isomorphism:
$$\Omega^q(NG(n)) \cong {\rm Inv}_G[ \Omega^0(PG(n)) \otimes  \Lambda^q \Sigma \mathcal{G}^*(n) ].$$
Before we consider the cohomology of the horizontal complex $H_{\delta}^*({\rm Inv}_G[k\{ \Omega^0(PG) \times  \Lambda^q \Sigma \mathcal{G}^* \}])$,
we observe the complex $\mathfrak{P}^q_{\delta}G:=k\{\Omega^0(PG(*)) \times \Lambda^q \Sigma \mathcal{G}^*(*)\}$.
\begin{lemma}
$$H_{\delta}(\Omega^0(PG(n))) \cong \begin{cases}
{\mathbb R} & (n=0)\\
0 & {\rm otherwise}
\end{cases},~~~~~
H_{\delta}(\Lambda^q \Sigma \mathcal{G}^*(n)) \cong \begin{cases}
S^q{\mathcal G}^* & (n=q)\\
0 & {\rm otherwise},
\end{cases}$$
So
$$H_{\delta}^n({\mathfrak P} ^qG)\ \cong \begin{cases}
S^q{\mathcal G}^* & (n=q)\\
0 & {\rm otherwise}.
\end{cases}$$\\
\end{lemma}

Since the cochain complex
$${\mathfrak P} ^qG:\Omega^0(PG(0))\otimes \Lambda^q \Sigma \mathcal{G}^*(0) \xrightarrow{{\delta}_0} \Omega^0(PG(1))\otimes \Lambda^q \Sigma \mathcal{G}^*(1) \xrightarrow{{\delta}_1} \cdots$$
is continuously injective, we obtain the following continuously injective resolution of $S^q{\mathcal G}^*$ from Lemma 3.1.
$$ S^q{\mathcal G}^*(={\rm Ker}{\delta}_q/{\rm Im}{\delta}_{q-1}) \xrightarrow{\delta_q} (\Omega^0(PG(q+1))\otimes \Lambda^q \Sigma \mathcal{G}^*(q+1))/{\rm Im}{\delta}_{q}  $$
$$ \xrightarrow{{\delta}_{q+1}} \Omega^0(PG(q+2))\otimes \Lambda^q \Sigma \mathcal{G}^*(q+2) \xrightarrow{{\delta}_{q+2}} \cdots~~~~~({\rm exact}).$$

Therefore  $H^{p}_c(G;S^q{\mathcal{G}^*})$ is equal to the $p$-th cohomology of the complex below.
$$   {\rm Inv}_G S^q {\mathcal{G}^*} \xrightarrow{\delta_q} {\rm Inv}_G [\Omega^0(PG(q+1))\otimes \Lambda^q \Sigma \mathcal{G}^*(q+1)/{\rm Im}{\delta}_{q}]$$
$$ \xrightarrow{{\delta}_{q+1}} {\rm Inv}_G [\Omega^0(PG(q+2))\otimes \Lambda^q \Sigma \mathcal{G}^*(q+2)] \xrightarrow{{\delta}_{q+2}}\cdots$$
So we obtain the following isomorphism.
$$ H^{p}_c(G;S^q{\mathcal{G}^*}) \cong H^{p+q}_{\delta}({\rm Inv}_G [k \{ \Omega^0(PG)\times \Lambda^q \Sigma \mathcal{G}^* \}]).$$
\end{proof}

\begin{corollary}[Bott,\cite{Bot}]
If $G$ is compact,
$$H^{p}_{\delta}(\Omega^q(NG)) \cong \begin{cases}
{\rm Inv}_GS^q\mathcal{G}^* & (p=q)\\
0 & {\rm otherwise.}
\end{cases}$$
\end{corollary}

\bigskip

\section{The triple complex on bisimplicial manifold}
In this section we construct a triple complex on a bisimplicial manifold.\\

A bisimplicial manifold is a sequence of manifolds with horizontal and vertical face and degeneracy operators which commute with each other.
A bisimplicial map is a sequence of maps commuting with horizontal and vertical face and degeneracy operators.
For a subgroup $H$ of $G$, we define a bisimplicial manifold $NG(*) \rtimes NH(*)$ as follows;
\par
$$NG(p) \rtimes NH(q)  := \overbrace{G \times \cdots \times G }^{p-times} \times \overbrace{H \times \cdots \times H }^{q-times}. $$  
Horizontal face operators \enspace ${\varepsilon}_{i}^{G} : NG(p) \rtimes NH(q) \rightarrow NG(p-1)  \rtimes NH(q) $ are the same as the face operators of $NG(p)$.
Vertical face operators \enspace ${\varepsilon}_{i}^{H} : NG(p) \rtimes NH(q) \rightarrow NG(p)  \rtimes NH(q-1) $ are
$$
{\varepsilon}_{i}^{H}(\vec{g}, h_1 , \cdots , h_q )=\begin{cases}
(\vec{g}, h_2 , \cdots , h_q )  &  i=0 \\
(\vec{g}, h_1 , \cdots ,h_i h_{i+1} , \cdots , h_q )  &  i=1 , \cdots , q-1 \\
(h_{q}\vec{g}h_{q} ^{-1}, h_1 , \cdots , h_{q-1} )  &  i=q.
\end{cases}
$$
Here $\vec{g}=(g_1, \cdots , g_p)$.

We define a bisimplicial map $\gamma_{\rtimes} : P{G}(p) \times P{H}(q) \rightarrow NG(p) \rtimes NH(q) $ as $ \gamma_{\rtimes} (\vec{\bar{g}}, \bar{h}_1, \cdots ,\bar{h}_{q+1} ) = (\bar{h}_{q+1}\gamma (\vec{\bar{g}}) \bar{h}^{-1} _{q+1} ,
 \gamma (\bar{h}_1, \cdots, \bar{h}_{q+1}))$.
Now we fix a semi-direct product operator $\cdot_{\rtimes}$of $G \rtimes H$ as $(g, h) \cdot_{\rtimes} (g', h') := (ghg'h^{-1} , hh')$,
then
$G \rtimes H$ acts $ PG(p) \times P{H}(q)$ by right as $(\vec{\bar{g}},\vec{\bar{h}})\cdot(g,h) = (h^{-1}\vec{\bar{g}}gh, \vec{\bar{h}}h)$
and $\parallel  \gamma_{\rtimes} \parallel$ is a model of $E(G \rtimes H) \rightarrow B(G \rtimes H)$.

\begin{definition}
For a bisimplicial manifold $NG(*) \rtimes NH(*)$, we have a triple complex as follows:

$${\Omega}^{p,q,r} (NG(*) \rtimes NH(*)) := {\Omega}^{r} (NG(p) \rtimes NH(q)) $$

Derivatives are:
$$ \delta_G := \sum _{i=0} ^{p+1} (-1)^{i} ({{\varepsilon}^G _{i}}) ^{*}  , \qquad  \delta_H := \sum _{i=0} ^{q+1} (-1)^{i} ({{\varepsilon}^H _{i}}) ^{*} \times (-1)^{p} $$
$$ d' :=  (-1)^{p+q} \times {\rm the \enspace exterior \enspace differential \enspace on \enspace }{ \Omega ^*(NG(p) \rtimes NH(q)) }.$$

\end{definition}

\begin{theorem}[\cite{Suz}]
 If $H$ is compact, there exist isomorphisms
$$ H({\Omega}_H ^{*} (NG))  \cong H({\Omega}^{*} (NG \rtimes NH)) \cong  H^{*} (B(G \rtimes H)).$$
Here ${\Omega}_H ^{*} (NG)$  means the total complex in subsection 2.2 and ${\Omega}^{*} (NG \rtimes NH)$ means the 
total complex of the triple complex.
\end{theorem}

\section{Main theorem}

\begin{theorem}
For any fixed $q$, 
$$H^{p+q}_{\delta}(\Omega^q(NG \rtimes NH)) \cong H^{p}_c(G \rtimes H ;S^q{\mathcal{G}^*} \otimes S^q{\mathcal{H}^*}).$$
Here $\delta:=\delta_G+\delta_H$.
\end{theorem}

\begin{proof}

We identify $\Omega^q(NG(n) \rtimes NH(m))$ with ${\rm Inv}_{G \rtimes H} [ \Omega^0(PG(n)) \otimes  \Lambda^q \Sigma \mathcal{G}^*(n) \otimes \Omega^0(PH(m)) \otimes  \Lambda^q \Sigma \mathcal{G}^*(m)  ]$.

Before we deal with the cohomology $H_{\delta}^*({\rm Inv}_{G \rtimes H}[k\{ \Omega^0(PG) \times  \Lambda^q \Sigma \mathcal{G}^* \times \Omega^0(PH) \times  \Lambda^q \Sigma \mathcal{G}^* \}])$,
we observe the total complex of the double complex 
$$\mathfrak{P}^q_{\delta_G}G \otimes \mathfrak{P}^q_{\delta_H}H=k\{\Omega^0(PG(*)) \times \Lambda^q \Sigma \mathcal{G}^*(*)\} \otimes k\{\Omega^0(PH(*)) \times \Lambda^q \Sigma \mathcal{H}^*(*)\}.$$

From Lemma 3.1, we obtain:

$$H_{\delta}^n({\mathfrak P} ^qG \otimes \mathfrak{P}^q H)\ \cong \begin{cases}
S^q{\mathcal G}^* \otimes S^q{\mathcal H}^*& (n=q)\\
0 & {\rm otherwise}.
\end{cases}$$\\

Since the total complex 
$$k_{\delta}({\mathfrak P} ^qG \times \mathfrak{P}^q H)(0) \xrightarrow{{\delta}_0}  k_{\delta}({\mathfrak P} ^qG \times \mathfrak{P}^q H)(1) \xrightarrow{{\delta}_1} \cdots$$
is continuously injective, we obtain the following continuously injective resolution of $S^q{\mathcal G}^* \otimes S^q{\mathcal H}^*$.
$$ S^q{\mathcal G}^* \otimes S^q{\mathcal H}^* (={\rm Ker}{\delta}_q/{\rm Im}{\delta}_{q-1})\xrightarrow{{\delta}_{q}}
k_{\delta}({\mathfrak P} ^qG \times \mathfrak{P}^q H)(q+1)/{\rm Im}{\delta}_{q}$$
$$\xrightarrow{{\delta}_{q+1}} k_{\delta}({\mathfrak P} ^qG \times \mathfrak{P}^q H)(q+2)\xrightarrow{{\delta}_{q+2}} \cdots ({\rm exact}).$$

Therefore  $H^{p}_c(G \rtimes H;S^q{\mathcal{G}^*} \otimes S^q{\mathcal{H}^*})$ is equal to the $p$-th cohomology of the complex below.
$$ {\rm Inv}_{G \rtimes H} (S^q{\mathcal{G}^*} \otimes S^q{\mathcal{H}^*}) \xrightarrow{{\delta}_{q}}  {\rm Inv}_{G \rtimes H} [k_{\delta}({\mathfrak P} ^qG \times \mathfrak{P}^q H)(q+1)/{\rm Im}{\delta}_{q}]$$
$$ \xrightarrow{{\delta}_{q+1}} {\rm Inv}_{G \rtimes H} [k_{\delta}({\mathfrak P} ^qG \times \mathfrak{P}^q H)(q+2)] \xrightarrow{{\delta}_{q+2}} \cdots$$
So we obtain the following isomorphisms.
$$ H^{p}_c(G \rtimes H ;S^q{\mathcal{G}^*} \otimes S^q{\mathcal{H}^*}) \cong H^{p+q}_{\delta}({\rm Inv}_{ G \rtimes H} [k_{\delta}({\mathfrak P} ^qG \times \mathfrak{P}^q H)]) $$
$$\cong H^{p+q}_{\delta}(\Omega^q(NG \rtimes NH)).$$
\end{proof}

\begin{corollary}
If $G$ is compact,
$$H^{p}_{\delta}(\Omega^q(NG \rtimes NG)) \cong \begin{cases}
{\rm Inv}_{G \rtimes G}(S^q\mathcal{G}^* \otimes S^q\mathcal{G}^*) & (p=q)\\
0 & {\rm otherwise.}
\end{cases}$$
\end{corollary}


\begin{thebibliography}{99}
\bibitem{Ber}N. Berline, E. Getzler, and M. Vergne, Heat Kernels and Dirac Operators, Grundlehren Math. Wiss. 298, Springer-Verlag, Berlin, 1992.

\bibitem{Bot}R. Bott, On the Chern-Weil homomorphism and the continuous cohomology of the Lie group, Adv. in Math. 11 (1973), 289-303.
\bibitem{Bot2}R. Bott, H. Shulman, J. Stasheff, On the de Rham Theory of Certain Classifying Spaces, Adv. in Math. 20 (1976), 43-56.
\bibitem{Car}H. Cartan, La transgression dans un groupe de Lie et dans un espace fibr\'{e} principal, Colloque de Topologie, CBRM Bruxelles, 1950, pp. 57-71.


\bibitem{Dup2}J.L. Dupont, Curvature and Characteristic Classes, Lecture Notes in Math. 640, Springer Verlag, 1978.

\bibitem{Get}E. Getzler, The equivariant Chern character for non-compact Lie groups, Adv. Math. 109(1994),no.1,88-107.
\bibitem{Ho}G. Hochschild and G. D. Mostow, Cohomology of Lie groups, Illinois J. Math. 6 (1962), 367-401.
\bibitem{Mos}M. Mostow and J. Perchick, Notes on Gel'fand-Fuks Cohomology and Characteristic Classes (Lectures by Bott).In Eleventh Holiday Symposium. New Mexico State University, December 1973.

\bibitem{Seg}G. Segal, Classifying spaces and spectral sequences. Inst.Hautes \'{E}tudes Sci.Publ.Math.No.34 1968 105-112.
\bibitem{Suz}N. Suzuki, The equivariant simplicial de Rham complex and the classifying space of a semi-direct product group. Math. J. Okayama Univ.  57 (2015), 123-128. 

\end{thebibliography}
\end{document}